\documentclass[12pt, a4paper, reqno]{amsart}
\usepackage[left=2.5cm,right=2.2cm,top=2.5cm,bottom=3.5cm]{geometry}
\usepackage{graphicx}
\usepackage[T2A]{fontenc}
\usepackage[utf8]{inputenc}
\usepackage[english]{babel}
\usepackage{amsmath, amssymb, amsthm, mathtools}
\usepackage[hidelinks]{hyperref}
\usepackage{tikz-cd, quiver}
\usetikzlibrary{babel}

\title{The Bloch--Ogus Theorem over DVR}
\author{Ivan Gaidai-Turlov}
\address{Department of Mathematics and Computer Science, St Petersburg University, Saint Petersburg, Russia}
\email{gtivansan@gmail.com}
\date{26.11.2025}

\newcommand{\AffSp}{\mathbb{A}}
\newcommand{\ProjSp}{\mathbb{P}}

\newcommand{\Zar}{\normalfont\text{Zar}}
\newcommand{\et}{\normalfont\text{\'et}}

\DeclareMathOperator{\spec}{Spec}
\DeclareMathOperator{\codim}{codim}
\DeclareMathOperator{\gys}{gys}
\DeclareMathOperator{\tr}{Tr}
\DeclareMathOperator{\E}{E}

\newtheorem{theorem}{Theorem}[section]
\newtheorem{definition}[theorem]{Defenition}
\newtheorem{lemma}[theorem]{Lemma}
\newtheorem{corollary}[theorem]{Corollary}

\counterwithin{theorem}{subsection}
\renewcommand{\thetheorem}{%
  \ifnum\value{subsection}>0
    \thesubsection.%
  \else
    \thesection.%
  \fi
  \arabic{theorem}%
}

\renewenvironment{proof}[1][]{\noindent{\bf Proof#1. }}{\vskip 7pt}

\makeatletter
\newcommand*\rel@kern[1]{\kern#1\dimexpr\macc@kerna}
\newcommand*\widebar[1]{%
  \begingroup
  \def\mathaccent##1##2{%
    \rel@kern{0.8}%
    \overline{\rel@kern{-0.8}\macc@nucleus\rel@kern{0.2}}%
    \rel@kern{-0.2}%
  }%
  \macc@depth\@ne
  \let\math@bgroup\@empty \let\math@egroup\macc@set@skewchar
  \mathsurround\z@ \frozen@everymath{\mathgroup\macc@group\relax}%
  \macc@set@skewchar\relax
  \let\mathaccentV\macc@nested@a
  \macc@nested@a\relax111{#1}%
  \endgroup
}
\makeatother

\begin{document}

\begin{abstract}
We prove the Bloch--Ogus Theorem for regular local rings geometrically regular over a discrete valuation ring. In particular, we prove the Bloch--Ogus Theorem for regular local rings of mixed characteristic that are essentially smooth over a discrete valuation ring.
\end{abstract}

\maketitle

\section{Introduction}

The purpose of this note is to prove the Bloch--Ogus theorem \cite{Bloch1974GerstensCA} in mixed characteristic. Our main result is the following theorem:

\begin{theorem}\label{bo-th}
    Let $S$ be a regular local ring (possibly of mixed characteristic), geometrically regular over some discrete valuation ring $R$ with residue field $k$. Let $p = \operatorname{char} k$, and let $\Lambda$ be a finite $r$-torsion $\operatorname{Gal}(R)$-module with $p \nmid r$. Let $U = \spec (S)$, and let $\eta = \spec(K)$ denote the spectrum of the fraction field $K$ of $S$. Then the following Cousin complex is exact:
    \[0 \rightarrow \operatorname{H}^*_{\et}(U, \Lambda)\rightarrow \operatorname{H}^*_{\et}(\eta, \Lambda) \rightarrow \bigoplus_{x \in U^{(1)}} \operatorname{H}^{*+1}_x(U, \Lambda) \rightarrow \bigoplus_{x \in U^{(2)}} \operatorname{H}^{*+2}_x(U, \Lambda) \rightarrow \dots\]
    % If $S$ is essentially smooth $R$-algebra, then there is a canonical isomorphism \[\operatorname{H}^{*+c}_{x}(U, \Lambda) \cong \operatorname{H}^{*-c}_{et}(x, \Lambda(-c))\] 
    % Note, that due to the Gabber's absolute purity theorem \cite{absolutePurity}, for each codimension $c$ point $x$ of the scheme $U$, there is a canonical isomorphism $\operatorname{H}^{*+c}_{x}(U, \Lambda) \cong \operatorname{H}^{*-c}_{et}(x, \Lambda(-c))$.
\end{theorem}

Note that the equicharacteristic case of the theorem (i.e. $\Lambda$ is a $\operatorname{Gal}(F)$-module for some field $F$ сontained in $R$) is covered by the Bloch--Ogus theorem and the result of Panin \cite{PanEqui}. See also \cite{colliot1997bloch} for a good exposition and Gabber's proof of the Bloch--Ogus theorem that does not use Poincar\'e duality. See Definition~\ref{geom-reg} for the notion of geometrically regular morphisms.

There are some partial results on this problem in mixed characteristic. In \cite{Schmidt2019ABT}, Schmidt and Strunk proved the Nisnevich-local version of the theorem. In \cite{geisser2004motivic}, Geisser deduced the result for $H_{\et}^s(-,\mu_r^{\otimes n})$ for $s \leqslant n$ from the Bloch--Kato conjecture. In \cite{lüders2024blochogustheorysmoothsemistable}, Lüders generalized the method of Bloch and Ogus to mixed characteristics. His result applies to excellent base DVRs, and the proof uses Poincaré duality. Our proof is independent of all these results.

\vskip 7pt \noindent {\bf Acknowledgements.} I warmly thank Ivan Alexandrovich Panin for posing the problem, numerous discussions over tea, and support. 

This work was supported by the Ministry of Science and Higher Education of the Russian Federation (agreement 075-15-2025-344 dated 29/04/2025 for Saint Petersburg Leonhard Euler International Mathematical Institute at PDMI RAS).

\section{Preliminaries}\label{preliminaries}

Throughout the paper, we fix the following notation. Let $R$ be a discrete valuation ring, and let $\pi$ be a uniformizing parameter. Let $V = \spec R = \{v, \theta\}$, where $v$ is the closed point and $\theta$ is the generic point. Let $k = k(v) = R/\pi R$ denote the residue field of characteristic $p$. For any scheme $X$ over $V$, denote by $X_v$ and $X_\theta$ the fibers of $X$ over $v$ and $\theta$, respectively. Define 
\[\E(X) := \bigoplus_{n, m \in \mathbb{Z}} \operatorname{H}_{\et}^n(X, p^*\Lambda(m)) \quad \text{and} \quad \E_Z(X) := \bigoplus_{n, m \in \mathbb{Z}} \operatorname{H}_{Z}^n(X, p^*\Lambda(m)),\] 
where $\Lambda$ is an étale-locally constant, finite, prime-to-$p$, $r$-torsion sheaf of abelian groups on $V_{\et}$, and $p \colon X \to V$ is the structure morphism. Also denote by $\Phi_X^c$ the inductive system of supports in $X$ of codimension at least $c$. We omit the subscript $X$ when it doesn’t lead to confusion.

Let $X, Y$ be $V$-schemes, and let $Z' \subset Z \subset X$ be closed subschemes. The cohomology theory $\E$ has the following properties:

\begin{enumerate}
    \item[(0)] \textit{Functoriality}. $\E_Z(X)$ is a contravariant functor with respect to morphisms of pairs.
    \item \textit{Localization}. There is a functorial exact sequence:  
    \[\dotsb \rightarrow \E_Z(X) \rightarrow \E_{Z - Z'}(X - Z') \xrightarrow{\partial} \E_{Z'}(X) \rightarrow \E_Z(X) \rightarrow \dotsb\]    
    \item \textit{\'Etale excision}. For any \'etale neighborhood $U$ of $Z \subset X$, the following map is an isomorphism:  
    \[\E_Z(X) \xrightarrow{\simeq} \E_Z(U)\]
    \item \textit{Homotopy invariance}. Shift by a $V$-rational vector induce the identity on $\E(\AffSp^n_V \times_V X)$.
    \item \textit{Purity}. If $X$ and $Z$ are regular and $Z$ is of pure codimension $c$ in $X$, then there is a natural Gysin isomorphism:  
    \[\operatorname{gys}\colon \E_{Z'}(Z) \xrightarrow{\simeq} \E_{Z'}(X)\]
    \item \textit{0-dimensional case}. If $X$ is a semi-local principal domain, $Z$ is a set of closed points, and $\eta$ is the generic point, then the localization sequence takes the form:  
    \[0 \rightarrow \E(X) \rightarrow \E(\eta) \xrightarrow{\partial} \E_Z(X) \rightarrow 0\]
    \item \textit{Trace structure}. For any finite flat morphism $\pi\colon Y \to X$ of degree $d$, there is a trace homomorphism $\tr_{\pi}\colon \E_{\pi^{-1}(Z)}(Y) \to \E_Z(X)$ such that the composition $\tr_{\pi} \circ \pi^*$ is multiplication by $d$.
    \item If $p>0$, $\E_Z(X)$ is uniquely $p$-divisible.
    \item \textit{Continuity}. Let $X^\alpha$ be a filtered inductive system of affine $V$-schemes, and let $X = \varprojlim X^\alpha$. Then the canonical homomorphism $\varinjlim \E(X^\alpha) \to \E(X)$ is an isomorphism.
\end{enumerate}

Properties 0-2 are standard. Property 3 is obvious, since any $V$-isomorphism induces the identity on coefficients. Property 4 with $Z' = Z$ is standard for $X, Z$ smooth over $V$, and in general is a consequence of Gabber's absolute purity theorem \cite{absolutePurity}. Property 4 for arbitrary $Z'$ is discussed in Section~\ref{PaninsMethod}. Property 5 is proved in \cite[Appendix B]{colliot1997bloch}. See \cite[\S~11]{panin2022movinglemmasmixedcharacteristic} for the construction of $\tr_{\pi}$. Property 7 is obvious, since multiplication by $p$ induces an isomorphism on prime-to-$p$ torsion sheaves. Property 8 is also standard; see \cite[Theorem II.6.3.2]{tamme2012introduction}.

Our proof uses only these properties, so the result holds for any cohomology theory that satisfies these conditions. Moreover, for the ``geometric'' case, we do not need absolute purity: purity for $X, Z$ smooth over $V$ is sufficient. Properties 6 and 7 are used only in the reduction to the case of the perfect infinite residue field.

Recall that for any equidimensional and Noetherian scheme $X$, the following Cousin complex is constructed in \cite{colliot1997bloch} using properties (0-2):
\[0 \rightarrow \bigoplus_{x \in X^{(0)}} \E_x(X) \rightarrow \bigoplus_{x \in X^{(1)}} \E_x(X) \rightarrow \dots \text{ , where } \E_x(X) := \varinjlim \E_{\overline{\{x\}}\cap U}(U)\]
There is an isomorphism 
\[\bigoplus_{x \in X^{(c)}} \E_x(X) \cong \varinjlim \E_{Z^c - Z^{c+1}}(X - Z^{c+1}),\]
where the limit is taken over pairs $Z^c \subset Z^{c+1}$ with $Z^i \subset X$ being a closed subset of codimension at least $i$. The boundary maps in the Cousin complex then arise as a limit of the boundary maps in the localization sequence of the triple $(X-Z^{c}, X-Z^{c+1}, X-Z^{c+2})$.It can easily be shown, using functoriality of the boundary map, that if the support extension maps $\operatorname{ext} \colon \E_{\Phi^{c+1}}(U) \rightarrow \E_{\Phi^{c}}(U)$ vanish, then the Cousin complex has only non-zero homology in the term $\bigoplus_{x \in X^{(0)}} \E_x(X)$, which equals $\E(X)$.

The proof of Theorem~\ref{bo-th} is structured as follows. We start with the ``geometric case'' (Theorem~\ref{geometricCase}), i.e., with a local scheme that is the localization of a $V$-smooth scheme $X$ at a point. In Section~\ref{reduction}, we reduce to the case of a perfect infinite residue field $k(v)$. In Section~\ref{geometricCaseSection}, the ``geometric case'' is proved by induction on the relative dimension over $V$. We use a version of Lindel--Ojanguren--Gabber's geometric presentation lemma over DVR, Theorem~\ref{geomPres}, proved by Panin and Stavrova \cite{panin2024constantcasegrothendieckserreconjecture}, reducing to the case of an open subscheme of $\AffSp^n_V$. In fact, we need a semi-local version of this lemma, so we provide the proof in Appendix A. Subsequently, we apply a cohomology class modification trick (Lemma~\ref{classMod}). For the convenience of the proof, Lemmas~\ref{injectivity} and~\ref{zeroMaps} are stated below. The general case is derived from the geometric case using Panin's method \cite{PanEqui}. The proof is essentially the same as Panin's original approach, but algebraic $K$-theory is replaced with the \'etale cohomology theory $\E$. In Section~\ref{PaninsMethod}, we outline the proof.

\begin{lemma}\label{injectivity} 
    Let $X$ be a smooth affine irreducible $V$-scheme. Let $x \subset X_v$ be a finite set of closed points, and let $U = \spec(\mathcal{O}_{X, x})$. Let $\eta$ denote the generic point of $U$. Then the map $\E(U) \rightarrow \E(\eta)$ is injective.
\end{lemma}

\begin{lemma}\label{zeroMaps}
    Let $U$ be as in Lemma~\ref{injectivity}. Then, for any $c \geqslant 0$, the following support extension map is zero:
    \[\operatorname{ext} \colon \E_{\Phi^{c+1}}(U) \to \E_{\Phi^{c}}(U).\]
\end{lemma}

\section{Reduction to the case of the perfect infinite residue field}\label{reduction}

% \begin{lemma}\label{Tr}
%     Let $\mathcal F$ be any etale locally constant sheaf on $Sch/V$. Then $\mathcal F$ admits a trace structure.  Specifically, for any finite flat morphism $\pi\colon X\rightarrow Y$ of degree $d$, there is a trace homomorphism $\tr_{\pi}\colon\mathcal F(X) \rightarrow \mathcal F(Y)$ such that the composition $\tr_{\pi}\circ\pi^*$ is multiplication by $d$.
% \end{lemma}
% \begin{proof}
%     See \cite{panin2022movinglemmasmixedcharacteristic}, $\mathsection 11$ for the construction of $\tr_{\pi}$.
% \end{proof}

\begin{lemma}\label{TrSupp}
    A trace structure exists on $\E_{\Phi^{c}}(-)$. Specifically, for any finite flat morphism $\pi\colon X \to Y$ of degree $d$, there is a trace homomorphism $\tr_{\pi}\colon \E_{\Phi^{c}}(X) \to \E_{\Phi^{c}}(Y)$ such that the composition $\tr_{\pi} \circ \pi^*\colon \E_{\Phi^{c}}(Y) \to \E_{\Phi^{c}}(Y)$ is multiplication by $d$.
\end{lemma}
\begin{proof}
    Since $\pi$ is finite and flat, for any closed subset $Z \subset Y$ of codimension at least $c$, its preimage $\pi^{-1}(Z) \subset X$ is also a closed subset of codimension at least $c$. Additionally, closed subsets of the form $\pi^{-1}(Z)$ are final in $\Phi^c$; see \cite[V.2.10]{AK}. Therefore, we have $\E_{\Phi^c}(X) = \varinjlim \E_{\pi^{-1}(Z)}(X)$. The trace structure on the coefficients naturally induces a homomorphism $\E_{\pi^{-1}(Z)}(X) \to \E_Z(Y)$. We then define $\tr_{\pi}$ as the limit of these homomorphisms.
\end{proof}

\begin{lemma}\label{dvrExt}
    Let $k'$ be a finite extension of the residue field $k$. Then there exists a DVR $R'$, finite and flat over $R$, with residue field $k'$ and uniformizer $\pi$.
\end{lemma}

\begin{proof}
    Let $k' = k[t]/(f)$ for some monic irreducible polynomial $f \in k[t]$. Choose a monic lift $F \in R[t]$ of $f$, and define $R' = R[t]/(F)$. Since $F$ is monic, $R'$ is finite and flat over $R$. We see that $R'/\pi R' \cong k'$, so $\pi R'$ is the unique maximal ideal of $R'$. Thus $R'$ is a DVR.
\end{proof}

\begin{lemma}\label{limitPhi}
    Let $\{A^\alpha\}$ be a filtered inductive system of Noetherian $R$-algebras with injective transition maps and Noetherian colimit $A$. Let $X^\alpha = \spec A^\alpha$, $X = \varprojlim X^\alpha$, and $A = \varinjlim A^\alpha$. Then $X = \spec A$, and there are canonical isomorphisms:
    \[\varinjlim \E(X^\alpha - \Phi^c) \to \E(X - \Phi^c) \quad \text{and} \quad \varinjlim \E_{\Phi^c}(X^\alpha) \to \E_{\Phi^c}(\varprojlim X^\alpha).\]
\end{lemma}
\begin{proof}
    Let $Z$ be a closed irreducible subscheme of $X$, and let $I$ be the corresponding ideal of $A$. Since $A$ is Noetherian, the ideal $I$ is finitely generated. Therefore, the generators of $I$ lie in some $A^\beta$ and generate an ideal $I^\beta \subset A^\beta$ such that $I^\beta A = I$. Let $Z^\beta \subset X^\beta$ be the corresponding closed subscheme. Then $Z$ is the preimage of $Z^\beta$ under the projection map $X \to X^\beta$. Since $Z$ is irreducible, $Z^\beta$ is also irreducible.

    Now, suppose $Z \subset X$ is a closed irreducible subscheme of codimension at least $c$. Then there exists a chain $Z = Z_0 \subsetneq Z_1 \subsetneq \dots \subsetneq Z_c \subset X$, where each $Z_i$ is a closed irreducible subset of $X$. For each $Z_i$, choose $\beta_i$ as constructed above. Fix $\beta$ such that $\beta > \beta_i$ for all $i$. This yields a chain $Z^\beta = Z_0^\beta \subsetneq Z_1^\beta \subsetneq \dots \subsetneq Z_c^\beta \subset X^\beta$, where $Z^\beta$ has codimension at least $c$. Similarly, for any closed subset $Z \subset X$ of codimension at least $c$, we construct $\beta = \beta(Z)$ such that $Z^\beta \subset X^\beta$ is a closed subscheme of codimension $c$ and $Z$ is the preimage of $Z^\beta$.

    Finally, we have
    \[\E((\varprojlim X^\alpha)-\Phi^c) = \varinjlim_{Z\in \Phi^c} \E((\varprojlim X^\alpha) - Z) = \varinjlim_{Z\in \Phi^c}\; \varinjlim_{\alpha > \beta(Z)} \E(X^\alpha-Z^\alpha)\]\
    \[\varinjlim \E(X^\alpha-\Phi^c) = \varinjlim_\alpha \; \varinjlim_{Z^\alpha\in \Phi^c} \E(X^\alpha - Z^\alpha)\]
    It is clear that these limits are isomorphic via the canonical homomorphism. This proves the first isomorphism. The second isomorphism follows immediately from the first isomorphism and the five-lemma.
\end{proof}

\begin{lemma}\label{reductionToInfinite}
    Assume that Lemma~\ref{zeroMaps} holds for any base scheme $V$ where the residue field $k$ is perfect and infinite. Then Lemma~\ref{zeroMaps} also holds for any base scheme $V$ where $k$ is finite or imperfect.
\end{lemma}
\begin{proof}
    % Suppose, that the residue field $k$ is finite. Fix a prime integer $l \neq p$. Let $k^\alpha/k$ be the unique field extension of degree $l^\alpha$ with $\alpha \in \mathbb Z_{>0}$. For each $\alpha$, Lemma $\ref{dvrExt}$ yields a d.v.r. $R^\alpha$ with the residue field $k^\alpha$ that is finite and flat extension of $R^{\alpha-1}$.
    Suppose that the residue field $k$ is finite. Let $k^\alpha/k$ be the unique field extension of degree $p^{2^\alpha}$ with $\alpha \in \mathbb{Z}_{>0}$. For each $\alpha$, Lemma~\ref{dvrExt} yields a DVR $R^\alpha$ with residue field $k^\alpha$ that is a finite and flat extension of $R^{\alpha-1}$.

    Suppose that the residue field $k$ is imperfect, and denote by $k'$ its perfect closure. Choose a well-ordering on $k'$. For each $\alpha \in k'$, this order defines a subfield $k^\alpha \subset k'$ generated by $k$ and all elements of $k'$ less than $\alpha$ with respect to the order. We construct $R^\alpha$ by transfinite induction on $\alpha$. If $\alpha = \beta + 1$, then $k^\alpha/k^\beta$ is a finite simple extension of $p$-primary degree. Lemma~\ref{dvrExt} then yields a DVR $R^\alpha$ with residue field $k^\alpha$ that is a finite and flat extension of $R^\beta$. If $\alpha$ is a limit ordinal, define $R^\alpha$ as the inductive limit $\bigcup_{\beta < \alpha} R^\beta$. Since each $R^\beta$ is a DVR with uniformizing parameter $\pi$, the colimit $R^\alpha$ is also a DVR with uniformizer $\pi$.
    
    Let $X$, $x \subset X_v$, and $U$ be as in Lemma~\ref{zeroMaps}. We introduce the following notation:
    \[V^\alpha = \spec(R^\alpha), \quad R' = \varinjlim R^\alpha, \quad V' = \spec(R') = \varprojlim V^\alpha\]
    \[U^\alpha = U \times_V V^\alpha, \quad U' = U \times_V V' = \varprojlim U^\alpha\]
    \[X^\alpha = X \times_V V^\alpha, \quad X' = X \times_V V' = \varprojlim X^\alpha\]
    The residue field $k'$ of $R'$ is perfect and infinite. Let $\pi\colon X' \to X$ be the projection. Then $\pi^{-1}(x)$ is a closed subscheme of $X'$ that is finite over the closed point of $V'$. Therefore, $\pi^{-1}(x)$ is a finite set of closed points of $X'_{v'}$, and $U' = \spec(\mathcal{O}_{X', \pi^{-1}(x)})$. Consider the following commutative diagram:
    % https://q.uiver.app/#q=WzAsNCxbMCwwLCJIX3tcXFBoaV57bisxfX0oVScpIl0sWzEsMCwiSF97XFxQaGlebn0oVScpIl0sWzAsMSwiSF97XFxQaGlee24rMX19KFUpIl0sWzEsMSwiSF97XFxQaGlebn0oVSkiXSxbMCwxLCJcXG9wZXJhdG9ybmFtZXtleHR9Il0sWzIsMywiXFxvcGVyYXRvcm5hbWV7ZXh0fSJdLFsyLDBdLFszLDEsIiIsMSx7InN0eWxlIjp7InRhaWwiOnsibmFtZSI6Imhvb2siLCJzaWRlIjoidG9wIn19fV1d
    \[\begin{tikzcd}
	{\E_{\Phi^{c+1}}(U')} & {\E_{\Phi^c}(U')} \\
	{\E_{\Phi^{c+1}}(U)} & {\E_{\Phi^c}(U)}
	\arrow["{\operatorname{ext}}", from=1-1, to=1-2]
	\arrow[from=2-1, to=1-1]
	\arrow["{\operatorname{ext}}", from=2-1, to=2-2]
	\arrow[hook, from=2-2, to=1-2]
    \end{tikzcd}\]
    Let us prove by transfinite induction that the map $\E_{\Phi^c}(U) \to \E_{\Phi^c}(U^\alpha)$ is injective, and in particular, that the vertical maps in the diagram are injective. If $\alpha = \beta + 1$, then $U^\alpha$ is flat and finite of $p$-primary degree over $U^\beta$. Recall that $\E_{\Phi^c}(U)$ is uniquely $p$-divisible. Therefore, $\frac{1}{p^d} \tr_{U^\alpha/U^\beta}$ is a section of the map $\E_{\Phi^c}(U^\beta) \to \E_{\Phi^c}(U^\alpha)$, implying injectivity. If $\alpha$ is a limit ordinal, then $\E_{\Phi^c}(U^\alpha) = \varinjlim_{\beta < \alpha} \E_{\Phi^c}(U^\beta)$ by Lemma~\ref{limitPhi}, and the desired map is injective as a filtered colimit of injective maps. Finally, the top horizontal arrow in the diagram vanishes by assumption, so the bottom one vanishes as well.
\end{proof}

\section{The geometric case}\label{geometricCaseSection}

\begin{lemma}\label{classExtension}
Let $n \in \mathbb Z_{>0}.$ Assume that lemma \ref{injectivity} holds for any smooth affine $V$-scheme $X$ of relative dimension less than $n$. Let $X$ be an irreducible smooth $V$-scheme of relative dimension n, and let $Y, Z \subsetneq X$ be closed subsets such that no irreducible component of Y is contained in Z. Given an element $a \in \E(X-Y)$ such that its image in $\E(X-Y-Z)$ is zero. Given a finite set of closed points $y \subset Y_v$, one in each irreducible component, such that $Y$ with its reduced induced scheme structure is smooth over $V$ at these points. Then there is a closed subset $W \subset Y$, and an element $\tilde a \in \E(X-W)$ such that $W \cap y = \varnothing$ and $\tilde a|_{X-Y} = a$.
\end{lemma}
\begin{proof}
Let $Y_0 \subset Y$  be an open affine neighborhood of points $y$ that is smooth over $V$. $W_0 = Y - Y_0$. By $\partial$ denote the boundary map in the long exact sequence of the pair $(X-W_0, X-Y)$. There is a Gysin isomorphism $\gys: \E(Y_0) \rightarrow \E_{Y_0}(X-W_0)$. Therefore we have an element $\gys^{-1}\partial a \in \E(Y_0)$. Since $a|_{X-Y-Z} = 0$, the image of $\gys^{-1}\partial a$ in $\E(Y_0-Z)$ is also zero. Note that $Y_0$ is a smooth $V$-scheme of relative dimension less than $n$. Thus by our assumption the map $\E(\spec (\mathcal O_{Y_0, y})) \rightarrow \bigoplus \E(\eta_i)$ is injective, where $\eta_i$ are the generic points of $Y_0$. It follows that $\gys^{-1}\partial a$ vanishes in some open neighborhood $Y-W$ of points $y$. Then W is desired and the long exact sequence of the pair $(X-W, X-Y)$ yields $\tilde a$.

% https://q.uiver.app/#q=WzAsMTIsWzEsMCwiSChYXFwhLVxcIVkpIl0sWzAsMCwiSChYXFwhLVxcIVcpIl0sWzIsMCwiSF97WS1XfShYXFwhLVxcIVkpIl0sWzMsMCwiSChZXFwhLVxcIVcpIl0sWzEsMSwiSChYXFwhLVxcIVkpIl0sWzIsMSwiSF97WV8wfShYXFwhLVxcIVdfMCkiXSxbMywxLCJIKFlfMCkiXSxbNCwxLCJIKFxcc3BlYyAoXFxtYXRoY2FsIE9fe1lfMCwgeX0pKSJdLFsxLDIsIkgoWFxcIS1cXCFZXFwhLVxcIVopIl0sWzIsMiwiSF97WV8wLVp9KFhcXCEtXFwhV18wXFwhLVxcIVopIl0sWzMsMiwiSChZXzBcXCEtXFwhWikiXSxbNCwyLCJcXGJpZ29wbHVzIEgoXFxldGFfaSkiXSxbNSwyXSxbNCwwLCJpZCJdLFsxLDBdLFs0LDUsIlxccGFydGlhbCJdLFswLDIsIlxccGFydGlhbCJdLFsyLDMsIlxcc2ltZXEiXSxbNSw2LCJcXHNpbWVxIl0sWzYsM10sWzYsN10sWzQsOF0sWzgsOSwiXFxwYXJ0aWFsIl0sWzksMTAsIlxcc2ltZXEiXSxbNiwxMF0sWzEwLDExXSxbNywxMSwiIiwwLHsic3R5bGUiOnsidGFpbCI6eyJuYW1lIjoiaG9vayIsInNpZGUiOiJ0b3AifX19XSxbNSw5XV0=
\[\begin{tikzcd}[column sep=small]
	{\E(X\!-\!W)} & {\E(X\!-\!Y)} & {\E_{Y-W}(X\!-\!Y)} & {\E(Y\!-\!W)} \\
	& {\E(X\!-\!Y)} & {\E_{Y_0}(X\!-\!W_0)} & {\E(Y_0)} & {\E(\spec (\mathcal O_{Y_0, y}))} \\
	& {\E(X\!-\!Y\!-\!Z)} & {\E_{Y_0-Z}(X\!-\!W_0\!-\!Z)} & {\E(Y_0\!-\!Z)} & {\bigoplus \E(\eta_i)}
	\arrow[from=1-1, to=1-2]
	\arrow["\partial", from=1-2, to=1-3]
	\arrow["\simeq", from=1-3, to=1-4]
	\arrow["\operatorname{id}", from=2-2, to=1-2]
	\arrow["\partial", from=2-2, to=2-3]
	\arrow[from=2-2, to=3-2]
	\arrow[from=2-3, to=1-3]
	\arrow["\simeq", from=2-3, to=2-4]
	\arrow[from=2-3, to=3-3]
	\arrow[from=2-4, to=1-4]
	\arrow[from=2-4, to=2-5]
	\arrow[from=2-4, to=3-4]
	\arrow[hook, from=2-5, to=3-5]
	\arrow["\partial", from=3-2, to=3-3]
	\arrow["\simeq", from=3-3, to=3-4]
	\arrow[from=3-4, to=3-5]
\end{tikzcd}\]
\end{proof}

\begin{lemma}\label{divisorExtension}
    Let $D_v \subset \AffSp^n_v$ be a reduced closed subscheme of pure codimension 1 that is smooth over $v$ at a finite set of points $x \subset \AffSp^n_v$. Given a finite set $z \subset \AffSp^n_V$ of codimension 1 points. Then there is a reduced closed subscheme $D \subset \AffSp^n_V$ of pure codimension 1 such that the closed fiber of $D$ over $v$ equals $D_v$, $D$ is smooth over $V$ at $x$, and $D\cap z=\varnothing$.
\end{lemma}
\begin{proof}
    % Say $D_v = Z(f)$ is a vanishing locus of $f \in k[T_1, \dots, T_n]$. Then $Z(f, f')$ is a singular locus of $D_v$ over $v$. Let $F \in R[T_1, \dots, T_n]$ be any lifting of $f$ such that $D = Z(F)$ contains no points of $z$. Since $Z(F, F')_v = Z(f, f')$, $D$ is smooth over $V$ at $x$.
    Say $D_v$ is given by a principal ideal $(f) \subset k[T_1, \dots, T_n]$. Let $F \in R[T_1, \dots, T_n]$ be any lifting of $f$ such that $D = \spec k[T_1, \dots, T_n]/(F)$ reduced and contains no points of $z$. The smoothness of $D/V$ at $x$ is automatic.
\end{proof}

\begin{lemma}\label{classMod}
    Let $n \in \mathbb Z_{>0}.$ Assume that lemma \ref{injectivity} holds for any smooth affine $V$-scheme $X$ of relative dimension less than $n$. Let $Y, Z \subsetneq \AffSp^n_V$ be closed subsets. Given an element $a \in \E(\AffSp^n_V-Y)$ such that its image in $\E(\AffSp^n_V-Y-Z)$ is zero. Given a finite set of closed points $x \subset X_v$ such that $Y \cap x=\varnothing$. Then there is a closed subset $\widetilde Y \subset \AffSp^n_V$, and an element $\tilde a \in \E(\AffSp^n_V - \widetilde Y)$ such that
    \begin{enumerate}
        \item $a$ and $\tilde a$ coincide in some open neighbourhood of $x$.
        \item $\codim(\widetilde Y_v, \AffSp^n_v)$ and $\codim(\widetilde Y_\theta, \AffSp^n_\theta)$ are at least 2.
    \end{enumerate}
\end{lemma}
\begin{proof}
    Without loss of generality, we can assume that no component of $Z$ is contained in $Y$. Let us extend $a$ by zero outside of Z. Therefore $a \in \E(\AffSp^n_V-(Y\cap Z))$. Replacing $Y$ with $Y\cap Z$, we can assume that $Y \subset Z$.  Denote by $\overline{Y_\theta}$ the closure of $Y_\theta$ in $\AffSp^n_V$. Choose a divisor $D_v \subset \AffSp^n_v$ such that $Y_v\subset D_v$ and $D_v\cap x =\varnothing$. Since the residue field $k$ is perfect, $D_v$ is generically smooth over $v$. Note that $\codim (\overline{Y_\theta}, \AffSp^n_V)\geqslant 2$. Consequently $\codim ({Y_\theta}, \AffSp^n_\theta) \geqslant 2$, $\codim ((\overline{Y_\theta})_v, \AffSp^n_v) \geqslant 2$ and therefore $D_v$ and $D_v-\overline{Y_\theta}$ have the same generic points. Choose a closed smooth point in each irreducible component of $D_v-\overline{Y_\theta}$. Lemma $\ref{divisorExtension}$ yields a divisor $D \subset \AffSp^n_V$ such that the closed fibre of $D$ over $v$ equals $D_v$, $D$ is smooth over $V$ at the chosen points, and no irreducible component of $D$ is contained in $Z$.

    Let us apply Lemma $\ref{classExtension}$ to the scheme $X = \AffSp^n_V-\overline{Y_\theta}$, its closed subsets $D-\overline{Y_\theta}$, $Z-\overline{Y_\theta}$ and an image of $a$ in $\E(\AffSp^n_V-\overline{Y_\theta}-D)$. This yields a closed subspace $W \subset D-\overline{Y_\theta}$ and an element $\tilde a \in \E(\AffSp^n_V-\overline{Y_\theta}-W)$. Define $\widetilde Y = \overline{Y_\theta}\cup W$. It remains to check that $\tilde a$ and $\widetilde Y$ are desired. Indeed, elements $a$ and $\tilde a$ coincide in ${\AffSp^n_V-\overline{Y_\theta}-D}$. Denote by $\overline W$ a closure of $W$ in $\AffSp^n_V$. Finally, $\codim((\overline W)_v, \AffSp^n_v) > \codim(D_v, \AffSp^n_v) \geqslant 1$ and $\codim((\overline W)_\theta, \AffSp^n_\theta) \geqslant \codim(\overline W, \AffSp^n_V) > \codim(D, \AffSp^n_V) \geqslant 1$.
\end{proof}

\begin{lemma}\label{finiteProjection}
    Let $S$ be an affine scheme and $t \colon S \rightarrow \ProjSp^n_S$ be any $S$-point of $\ProjSp^n_S$. Given a closed subscheme $M \subset \ProjSp^n_S-t(S)$. Denote by $pr\colon \ProjSp^n_S-t(S)\rightarrow \ProjSp^{n-1}_S$ a projection from $t$. Then $pr|_M: M\rightarrow \ProjSp^{n-1}_S$ is finite morphism.
\end{lemma}

\begin{lemma}\label{pointExist}
    Let $U$ be an open subscheme of $\ProjSp^n_V$ that intersects the closed fiber $\ProjSp^n_v$. Since the residue field $k=k(v)$ is infinite, there exists $V$-point of $U$.
\end{lemma}

\begin{lemma}\label{cylinder}
    Let $Y \subset \AffSp^n_V$ be a closed subset such that $\codim(Y_v, \AffSp^n_v)$ and $\codim(Y_\theta, \AffSp^n_\theta)$ are at least 2. Let $x \subset \AffSp^n_v-Y$ be a finite set of closed points, and let $Z \subset \AffSp^n_V$ be a closed subset that does not contain the closed fiber. Then there exists an isomorphism $\AffSp^n_V\cong \AffSp^1_V\times_V\AffSp^{n-1}_V$, where the projection onto the second factor is denoted by $\operatorname{pr}$, and there exists an open subscheme $U \subset \AffSp^{n-1}_V$ of the second factor such that
    \begin{enumerate}
        \item $\operatorname{pr}|_Y: Y \rightarrow \AffSp^{n-1}_V$ is a finite morphism. 
        \item $U = \AffSp^{n-1}_V - \operatorname{pr}(Y)$.
        \item $\operatorname{pr}(x) \subset U$ 
        \item $\operatorname{pr}^{-1}(\operatorname{pr}(x)) \not\subset Z$
    \end{enumerate}
\end{lemma}
\begin{proof}
    For each $V$-point of the scheme $\ProjSp^{n-1}_V$, there is a projection $\operatorname{pr}\colon \AffSp^n_V \rightarrow\AffSp^{n-1}_V$ from that point, and there is a corresponding isomorphism $\AffSp^n_V\cong \AffSp^1_V\times_V\AffSp^{n-1}_V$. Our goal is to find an open subscheme of $\ProjSp^{n-1}_V$ that intersects $\ProjSp^{n-1}_v$ and such that $V$-points of this subscheme satisfy conditions 1-4. The result then follows from lemma $\ref{pointExist}$.

    Let us introduce some notation. For any closed subscheme $X\subset\AffSp^n_V$, denote by $\overline X$ its closure in the projective space $\ProjSp^n_V$. For any $V$-point $p\in\ProjSp^{n-1}_V(V)$, denote by $p_v\in\ProjSp^{n-1}_v(v)$ the corresponding $v$-point. There are corresponding projections $\operatorname{pr}\colon \AffSp^n_V \rightarrow\AffSp^{n-1}_V$ and $\operatorname{pr}_v\colon \AffSp^n_v \rightarrow\AffSp^{n-1}_v$.

    Condition 1. By lemma $\ref{finiteProjection}$, for any $V$-point of $\ProjSp^{n-1}_V$ that does not intersect $\overline Y$, the condition is satisfied. It remains to check that $\overline Y$ does not contain $\ProjSp^{n-1}_v$. By assumption, $\codim(Y, \AffSp_V) \geqslant 2$. Note that $\ProjSp^{n-1}_v \subset \ProjSp^n_V$ is also of codimension 2 and irreducible. But no irreducible component of $Y$ is equal to $\ProjSp^{n-1}_v$.

    Condition 2. Since $\operatorname{pr}|_Y$ is a finite morphism, $\operatorname{pr}(Y)$ is closed. Then $U = \AffSp^{n-1}_V - \operatorname{pr}(Y)$ is open.

    Condition 3. Without loss of generality we can assume that $x$ consists of just one closed point. That is because we can satisfy the condition for each point independently. Since $\operatorname{pr}_v(x) = \operatorname{pr}(x)$ and $\operatorname{pr}_v(Y_v) = (\operatorname{pr}(Y))_v$, it is sufficient to check that $\operatorname{pr}_v(x) \notin \operatorname{pr}_v(Y_v)$. Denote by $k(w)$ Galois closure of the residue field $k(x)$ of the point $x$, and denote $w = \spec(k(w))$. Note that $k(x)$, and accordingly $k(w)$, are finite extensions of $k=k(v)$ since $x$ is closed. Denote by $\operatorname{pr}_w, x_w$ and $Y_w$ the base change of $\operatorname{pr}_v, x$ and $Y_v$ along $w \rightarrow v$. Also denote by $p_w\in\ProjSp^{n-1}_w(w)$ the $w$-rational point which corresponds to $p_v$. Obviously, $\operatorname{pr}_w$ is a projection from $p_w$. By the choice of $w$,  $x_w$ is a finite union of $w$-rational points $x_{w,i}$. From $x \notin Y$ we conclude that $x_w \cap Y_w =\varnothing$. Therefore, by lemma $\ref{finiteProjection}$ there is a finite projection $\operatorname{pr}_{i}: \overline {Y_w} \rightarrow \ProjSp^{n-1}_w$, where $\overline{Y_w}$ is a closure of $Y_w$ in $\ProjSp^n_w$. Consider line $l_i \subset \ProjSp^n_w$ between points $x_{w,i}$ and $p_w$. Then $\operatorname{pr}_w(x_{w,i}) \in \operatorname{pr}_w(Y_w) \Rightarrow \overline{Y_w} \cap l_i \neq\varnothing \Rightarrow p_w \in \operatorname{pr}_{i}(\overline{Y_w})$. Assume that $\operatorname{pr}_v(x) \in \operatorname{pr}_v(Y_v)$. Then $\operatorname{pr}_w(x_w) \cap \operatorname{pr}_w(Y_w) \neq\varnothing$ and for some index $i$ we have $p_w \in \operatorname{pr}_{i}(\overline{Y_w})$. Denote by $\pi$ the finite projection $\ProjSp^n_w \rightarrow \ProjSp^n_v$. Finally, $p_v = \pi(p_w) \in \pi(\operatorname{pr}_{i}(\overline{Y_w}))$. It is sufficient to check that $\bigcup_i \pi(\operatorname{pr}_{i}(\overline{Y_w})) \neq \ProjSp^{n-1}_v$, but this is from the dimension argument. $\dim (\pi(\operatorname{pr}_{i}(\overline{Y_w}))) = \dim (\operatorname{pr}_{i}(\overline{Y_w})) = \dim (\overline{Y_w}) = \dim Y_w = \dim Y_v \leqslant n-2$.

    Condition 4. Assume that $\operatorname{pr}^{-1}(\operatorname{pr}(x)) \subset Z$. Then, obviously, $\operatorname{pr}^{-1}(\operatorname{pr}(x)) \subset Z_v$ and $p_v \in \overline{Z_v}$, where $\overline{Z_v}$ is the closure of $Z_v$ in the $\ProjSp^n_v$. $\overline{Z_v}$ does not contain $\ProjSp^{n-1}_v$ since its irreducible components are of dimension $\leqslant n-1$ and do not equal $\ProjSp^{n-1}_v$.
\end{proof}

\begin{corollary}\label{injCor}
    Under the conditions of Lemma~\ref{cylinder}, any element of $\E(\AffSp^n_V - Y)$ that vanishes in $\E(\AffSp^n_V - Y - Z)$ also vanishes in some open neighborhood of $x \subset \AffSp^n_V - Y$.
\end{corollary}
\begin{proof}
    By condition (4) of Lemma~\ref{cylinder}, there exists a shift $\mathrm{T}\colon \AffSp^1_V \times_V U \to \AffSp^1_V \times_V U$ along $\AffSp^1_V$ such that $x \cap \mathrm{T}(Z) = \varnothing$. Since $\E$ is homotopy invariant, we have the commutative diagram:
    % https://q.uiver.app/#q=WzAsNCxbMCwwLCJIKFxcQWZmU3BeMV9WXFx0aW1lc19WIFUpIl0sWzEsMCwiSChcXEFmZlNwXjFfVlxcdGltZXNfViBVKSJdLFswLDEsIkgoXFxBZmZTcF4xX1ZcXHRpbWVzX1YgVSAtIFopIl0sWzEsMSwiSChcXEFmZlNwXjFfVlxcdGltZXNfViBVIC0gVChaKSkiXSxbMCwxLCJcXG9wZXJhdG9ybmFtZXtpZH0iXSxbMCwyXSxbMSwzXSxbMiwzLCJUIl1d
    \[\begin{tikzcd}
        {\E(\AffSp^1_V\times_V U)} & {\E(\AffSp^1_V\times_V U)} \\
        {\E(\AffSp^1_V\times_V U - Z)} & {\E(\AffSp^1_V\times_V U - \operatorname{T}(Z))}
        \arrow["{\operatorname{id}}", from=1-1, to=1-2]
        \arrow[from=1-1, to=2-1]
        \arrow[from=1-2, to=2-2]
        \arrow["\operatorname{T}", from=2-1, to=2-2]
    \end{tikzcd}\]
    Then $\AffSp^1_V \times_V U - \mathrm{T}(Z)$ is the desired open neighborhood.
\end{proof}

\begin{corollary}\label{zeroMapCor}
    Under the conditions of Lemma~\ref{cylinder}, any element of $\E_Z(\AffSp^n_V - Y)$ vanishes in some neighborhood of $x \subset \AffSp^n_V - Y$ after support extension to $\operatorname{pr}^{-1}(\overline{\operatorname{pr}(Z)})$.
\end{corollary}

\begin{proof}
    By Condition~(4) of Lemma~\ref{cylinder}, there exists a shift $\mathrm{T}\colon \AffSp^1_V \times_V U \to \AffSp^1_V \times_V U$ along $\AffSp^1_V$ such that $x \cap \mathrm{T}(Z) = \varnothing$. Let $Z' = \operatorname{pr}^{-1}(\overline{\operatorname{pr}(Z)})$ and $W = \AffSp^1_V \times_V U - \mathrm{T}(Z)$. Since $\mathrm{T}(Z') = Z'$ and $\E$ is homotopy invariant, we have the commutative diagram:
    % https://q.uiver.app/#q=WzAsNixbMCwxLCJIX3tUKFopfShcXEFmZlNwX1ZeMSBcXHRpbWVzX1YgVSkiXSxbMSwxLCJIX3taJ30oXFxBZmZTcF9WXjEgXFx0aW1lc19WIFUpIl0sWzAsMiwiSF97Wn0oXFxBZmZTcF9WXjEgXFx0aW1lc19WIFUpIl0sWzEsMiwiSF97Wid9KFxcQWZmU3BfVl4xIFxcdGltZXNfViBVKSJdLFsxLDAsIkhfe1onIFxcY2FwIFd9KFcpIl0sWzAsMCwiSF97VChaKSBcXGNhcCBXfShXKSJdLFsyLDNdLFszLDEsImlkIl0sWzIsMCwiVCJdLFswLDFdLFsxLDRdLFswLDVdLFs1LDRdXQ==
    \[\begin{tikzcd}
        {\E_{\operatorname{T}(Z) \cap W}(W)} & {\E_{Z' \cap W}(W)} \\
        {\E_{\operatorname{T}(Z)}(\AffSp_V^1 \times_V U)} & {\E_{Z'}(\AffSp_V^1 \times_V U)} \\
        {\E_{Z}(\AffSp_V^1 \times_V U)} & {\E_{Z'}(\AffSp_V^1 \times_V U)}
        \arrow[from=1-1, to=1-2]
        \arrow[from=2-1, to=1-1]
        \arrow[from=2-1, to=2-2]
        \arrow[from=2-2, to=1-2]
        \arrow["\operatorname{T}", from=3-1, to=2-1]
        \arrow[from=3-1, to=3-2]
        \arrow["\operatorname{id}", from=3-2, to=2-2]
    \end{tikzcd}\]
    Note that $\mathrm{T}(Z) \cap W = \varnothing$. Hence, the composition $\E_{Z}(\AffSp_V^1 \times_V U) \to \E_{Z' \cap W}(W)$ vanishes.
\end{proof}

\begin{lemma}\label{reductionToAff}
    To prove Lemma~\ref{zeroMaps} in dimension $n$, it suffices to consider the case $X = \AffSp^n_V$.
\end{lemma}
\begin{proof}
    We keep the notation of Lemma~\ref{zeroMaps}. Any element $a \in \E_{\Phi^{c+1}}(U)$ is the image of an element in $\E_Z(W)$, where $Z$ is a closed subset of codimension $\geqslant c+1$ and $W$ is an open affine neighborhood of $x$. Assume $Z$ contains no irreducible component of the closed fiber. Note that this holds automatically if $c \geqslant 1$. By Theorem~\ref{geomPres}, after replacing $Z$ with $Z \cap \mathcal{X}$, we obtain an elementary distinguished square:
    % https://q.uiver.app/#q=WzAsNCxbMSwwLCJcXG1hdGhjYWwgWCJdLFswLDAsIlxcbWF0aGNhbCBYIC0gWiJdLFsxLDEsIlxcbWF0aGNhbCBXIl0sWzAsMSwiXFxtYXRoY2FsIFctWiJdLFswLDIsIlxcdGF1Il0sWzEsMCwiIiwwLHsic3R5bGUiOnsidGFpbCI6eyJuYW1lIjoiaG9vayIsInNpZGUiOiJ0b3AifX19XSxbMywyLCIiLDIseyJzdHlsZSI6eyJ0YWlsIjp7Im5hbWUiOiJob29rIiwic2lkZSI6InRvcCJ9fX1dLFsxLDNdXQ==
    \[\begin{tikzcd}
        {\mathcal X-Z} & {\mathcal X} \\
        {\mathcal A-\tau(Z)} & {\mathcal A}
        \arrow[hook, from=1-1, to=1-2]
        \arrow[from=1-1, to=2-1]
        \arrow["\tau", from=1-2, to=2-2]
        \arrow[hook, from=2-1, to=2-2]
    \end{tikzcd}\]
    Clearly, $a$ is the image of an element in $\E_Z(\mathcal{X})$. By \'etale excision, it is also the image of some $\tilde{a} \in \E_{\tau(Z)}(\mathcal{A})$. By assumption, there exists a closed subset $Z' \subset \mathcal{A}$ such that $\tilde{a}$ vanishes after support extension to $Z'$. Hence, $a$ vanishes after support extension to $\tau^{-1}(Z')$. Since $\tau$ is \'etale, $\tau^{-1}(Z') \subset \mathcal{X}$ is closed of codimension $\geqslant c$.

     For $c = 0$, we show the map $\E(U) \to \E(U_{(v)})$ is injective, where $U_{(v)}$ denotes the localization of $U$ at the generic points of $U_v$. Indeed, $\E_{\widetilde{\Phi}^1}(U) \to \E(U)$ is zero, where $\widetilde{\Phi}^1$ is the system of supports of codimension $\geqslant 1$ that do not contain any irreducible component of the closed fiber. The scheme $U_{(v)}$ is 1-dimensional and regular, so $\E(U_{(v)}) \to \E(\eta)$ is also injective by \cite[Appendix B]{colliot1997bloch}.
\end{proof}

\begin{proof}[ of Lemma~\ref{injectivity}]
    We keep the notation of Lemma~\ref{injectivity}. The proof proceeds by induction on the relative dimension of $X$ over $V$. The case of relative dimension 0 is established in \cite[Appendix B]{colliot1997bloch}. Assume Lemma~\ref{injectivity} holds for all smooth affine $V$-schemes of relative dimension $< n$. By Corollary~\ref{reductionToAff}, it suffices to consider $X = \AffSp^n_V$.

    Any element of $\E(U)$ is the image of some $a \in \E(\AffSp^n_V - Y)$, where $Y \subset \AffSp^n_V$ is a closed subset disjoint from $x$. Suppose $a$ lies in the kernel of $\E(U) \to \E(\eta)$. Then there exists a closed subset $Z \subset \AffSp^n_V$ such that the image of $a$ in $\E(\AffSp^n_V - Y - Z)$ vanishes. By Lemma~\ref{classMod}, there is a closed subset $\widetilde{Y} \subset \AffSp^n_V$, fiberwise of codimension $\geqslant 2$, and an element $\tilde{a} \in \E(\AffSp^n_V - \widetilde{Y})$ representing the same class as $a$ in $\E(U)$. Furthermore, there exists a closed subset $\widetilde{Z} \subset \AffSp^n_V$ such that the image of $\tilde{a}$ in $\E(\AffSp^n_V - \widetilde{Y} - \widetilde{Z})$ vanishes, and $\widetilde{Z}$ can be chosen not to contain the closed fiber. Indeed, let $U_{(v)}$ denote the localization of $U$ at the generic points of $U_v$. Since $U_{(v)}$ is a 1-dimensional regular scheme, the map $\E(U_{(v)}) \to \E(\eta)$ is injective by \cite[Appendix B]{colliot1997bloch}. Corollary~\ref{injCor} then completes the proof.
\end{proof}

\begin{proof}[ of lemma \ref{zeroMaps}]
    We keep the notation of Lemma~\ref{zeroMaps}. The proof proceeds by induction on the relative dimension of $X$ over $V$. The case $c = 0$ (and accordingly relative dimension 0) is covered by Lemma~\ref{injectivity}. Hence, assume $c \geqslant 1$ and that Lemma~\ref{zeroMaps} holds for all smooth affine $V$-schemes of relative dimension $< n$. By Corollary~\ref{reductionToAff}, it suffices to consider $X = \AffSp^n_V$. To establish the lemma, we verify the following effaceability property:

    For any open neighborhood $W \subset X$ of $x$, closed subset $Z \subset X$ of codimension $\geqslant c+1$, and element $a \in \E_Z(W)$, there exist an open neighborhood $W' \subset W$ of $x$ and a closed subset $Z' \subset X$ of codimension $\geqslant c$ containing $Z$ such that the image of $a$ in $\E_{Z'}(W')$ vanishes.

    Let $W = \AffSp^n_V - Y$. We may assume no irreducible component of $Z$ is contained in $Y$. The excision isomorphism  
    \[\E_Z(\AffSp^n_V - Y) \cong \E_Z(\AffSp^n_V - (Y \cap Z))\]  
    applies. Since $Y \cap Z$ has codimension $\geqslant c + 2 \geqslant 3$ (hence fiberwise codimension $\geqslant 2$), Corollary~\ref{zeroMapCor} completes the proof.
\end{proof}

\begin{theorem}\label{geometricCase}
    Let $X$ be a smooth $V$-scheme, and let $y \in X$ be a point (possibly non-closed). Denote by $U = \spec(\mathcal{O}_{X, y})$ the local scheme at $y$. Let $\underline{\E}$ be the Zariski sheaf associated with $\E$. Then the following complex of Zariski sheaves on $X$ is exact:
    \[0 \rightarrow \underline{\E} \rightarrow \bigoplus_{x \in X^{(0)}} i_{x,*}\E_x(X) \rightarrow \bigoplus_{x \in X^{(1)}} i_{x,*}\E_x(X) \rightarrow \dotsb \tag{$*$}\]
    In particular, the complex of stalks at $y$ is exact:
    \[0 \rightarrow \E(U) \rightarrow \bigoplus_{x \in U^{(0)}} \E_x(U) \rightarrow \bigoplus_{x \in U^{(1)}} \E_x(U) \rightarrow \dotsb\]
\end{theorem}
\begin{proof}
    As noted in Section~\ref{preliminaries}, Lemma~\ref{zeroMaps} implies the exactness of $(*)$ at closed points of $X_v$. The exactness at points of $X_\theta$ follows from the Bloch–Ogus–Gabber theorem \cite[Prop. 2.1.2, Th. 2.2.1]{colliot1997bloch}. The following easy lemma then completes the proof.
\end{proof}

\begin{lemma}
    Let $X \to Y$ be a morphism of finite type, and let $\mathcal{F}_\bullet$ be a complex of abelian Zariski sheaves on $X$. If $\mathcal{F}_\bullet$ is exact at points of $X$ that are closed in their fibers over $Y$, then $\mathcal{F}_\bullet$ is exact.
\end{lemma}

\section{The general case}\label{PaninsMethod}

In this section, by an essentially $V$-smooth scheme, we mean a localization of an affine $V$-smooth scheme. Recall the definition of a geometrically regular morphism:

\begin{definition}\label{geom-reg}
    A morphism of schemes $f \colon X \to Y$ is called geometrically regular if it is flat with locally Noetherian fibers such that for any point $y \in Y$ and any finitely generated field extension $L/k(y)$, the fiber over $\spec(L)$ is regular. In this case, we say $X$ is geometrically regular over $Y$.
\end{definition}

Note that such morphisms are often called simply "regular" in the literature, while the term "geometrically regular" is typically reserved for schemes over fields. To avoid confusion with the abstract notion of a regular scheme, we use "geometrically regular" for all such morphisms.

\begin{lemma}\label{GysBoundaryCommute}
    Let $i_{Y,X}\colon Y \to X$ and $i_{Z,Y}\colon Z \to Y$ be closed embeddings of schemes. Assume that $Y$ is of pure codimension $c$ and that purity holds for the pair $(Y, X)$ and a sheaf $\mathcal{F}$ on $X_{\et}$, i.e., $R^{2c}i_{Y,X}^!\mathcal{F} \cong \mathcal{F}(-c)$. Then there is a natural isomorphism  
    \[\operatorname{gys}\colon \operatorname{H}^{*-2c}_Z(Y, \mathcal{F}(-c)) \to \operatorname{H}^*_Z(X, \mathcal{F})\]  
    that makes the following diagram commutative.
    % https://q.uiver.app/#q=WzAsNixbMCwwLCJcXG9wZXJhdG9ybmFtZXtIfV57Kn1fe1ktWn0oWC1aLCBcXG1hdGhjYWwgRikiXSxbMSwwLCJcXG9wZXJhdG9ybmFtZXtIfV57KisxfV97Wn0oWCwgXFxtYXRoY2FsIEYpIl0sWzAsMSwiXFxvcGVyYXRvcm5hbWV7SH1eeyotMmN9KFktWiwgXFxtYXRoY2FsIEYoLWMpKSJdLFsxLDEsIlxcb3BlcmF0b3JuYW1le0h9XnsqKzEtMmN9X1ooWSwgXFxtYXRoY2FsIEYoLWMpKSJdLFsyLDAsIlxcb3BlcmF0b3JuYW1le0h9XnsqKzF9X3tZfShYLCBcXG1hdGhjYWwgRikiXSxbMiwxLCJcXG9wZXJhdG9ybmFtZXtIfV57KisxLTJjfShZLCBcXG1hdGhjYWwgRigtYykpIl0sWzIsMCwiXFxvcGVyYXRvcm5hbWV7Z3lzfSIsMl0sWzIsMywiXFxwYXJ0aWFsIl0sWzAsMSwiXFxwYXJ0aWFsIiwwLHsic2hvcnRlbiI6eyJzb3VyY2UiOjEwLCJ0YXJnZXQiOjEwfX1dLFszLDEsIlxcb3BlcmF0b3JuYW1le2d5c30iLDJdLFszLDEsIlxcc2ltZXEiLDMseyJvZmZzZXQiOi0yLCJzdHlsZSI6eyJib2R5Ijp7Im5hbWUiOiJub25lIn0sImhlYWQiOnsibmFtZSI6Im5vbmUifX19XSxbMiwwLCJcXHNpbWVxIiwzLHsib2Zmc2V0IjotMiwic3R5bGUiOnsiYm9keSI6eyJuYW1lIjoibm9uZSJ9LCJoZWFkIjp7Im5hbWUiOiJub25lIn19fV0sWzEsNCwiIiwwLHsic2hvcnRlbiI6eyJzb3VyY2UiOjEwLCJ0YXJnZXQiOjEwfX1dLFszLDVdLFs1LDQsIlxcb3BlcmF0b3JuYW1le2d5c30iLDJdLFs1LDQsIlxcc2ltZXEiLDMseyJvZmZzZXQiOi0yLCJzdHlsZSI6eyJib2R5Ijp7Im5hbWUiOiJub25lIn0sImhlYWQiOnsibmFtZSI6Im5vbmUifX19XV0=
    \[\begin{tikzcd}
        {\operatorname{H}^{*}_{Y-Z}(X-Z, \mathcal F)} & {\operatorname{H}^{*+1}_{Z}(X, \mathcal F)} & {\operatorname{H}^{*+1}_{Y}(X, \mathcal F)} \\
        {\operatorname{H}^{*-2c}_{\et}(Y-Z, \mathcal F(-c))} & {\operatorname{H}^{*+1-2c}_Z(Y, \mathcal F(-c))} & {\operatorname{H}^{*+1-2c}_{\et}(Y, \mathcal F(-c))}
        \arrow["\partial", shorten <=3pt, shorten >=3pt, from=1-1, to=1-2]
        \arrow[shorten <=4pt, shorten >=4pt, from=1-2, to=1-3]
        \arrow["{\operatorname{gys}}"', from=2-1, to=1-1]
        \arrow["\simeq"{marking, allow upside down}, shift left=2, draw=none, from=2-1, to=1-1]
        \arrow["\partial", from=2-1, to=2-2]
        \arrow["{\operatorname{gys}}"', from=2-2, to=1-2]
        \arrow["\simeq"{marking, allow upside down}, shift left=2, draw=none, from=2-2, to=1-2]
        \arrow[from=2-2, to=2-3]
        \arrow["{\operatorname{gys}}"', from=2-3, to=1-3]
        \arrow["\simeq"{marking, allow upside down}, shift left=2, draw=none, from=2-3, to=1-3]
    \end{tikzcd}\]
\end{lemma}
\begin{proof}
    Recall that for a closed embedding $i\colon Z \to X$, the functor $i^!$ preserves injectives. This yields a spectral sequence:
    \[E^{p,q} = \operatorname{H}^p_{\et}(X, R^qi^!\mathcal{F}) \Rightarrow \operatorname{H}^{p+q}_Z(X, \mathcal{F})\]
    When purity holds for the pair $(Z, X)$ and a sheaf $\mathcal{F}$, this spectral sequence degenerates at the $E_2$-page, defining a Gysin isomorphism:
    \[\operatorname{H}^p_{\text{\'et}}(X, R^{2c}i^!\mathcal{F}) \xrightarrow{\simeq} \operatorname{H}^{p+2c}_Z(X, \mathcal{F})\]
    Consider the diagram with commutative top and bottom triangles:
    % https://q.uiver.app/#q=WzAsNixbMCwxLCJEKFgpIl0sWzEsMSwiRChZKSJdLFsyLDEsIkQoXFxtYXRoYmZ7QWJ9KSJdLFswLDAsIkQoWC1aKSJdLFsxLDAsIkQoWS1aKSJdLFsyLDAsIkQoXFxtYXRoYmZ7QWJ9KSJdLFsxLDIsIlJcXEdhbW1hXllfWlsxXSJdLFswLDEsIlJpXiFfe1ksWH0iXSxbNCw1LCJSXFxHYW1tYV97WS1afSIsMl0sWzMsNCwiUmleIV97WVxcIS1cXCFaLFhcXCEtXFwhWn0iLDJdLFswLDMsImpfe1osWH1eKiIsMl0sWzEsNCwial97WixZfV4qIiwyXSxbMiw1LCJcXG9wZXJhdG9ybmFtZXtpZH0iLDJdLFszLDUsIlJcXEdhbW1hXntYLVp9X3tZLVp9IiwwLHsiY3VydmUiOi0yfV0sWzAsMiwiUlxcR2FtbWFeWF9aWzFdIiwyLHsiY3VydmUiOjJ9XV0=
    \[\begin{tikzcd}[column sep=large]
        {D(X-Z)} & {D(Y-Z)} & {D(\mathbf{Ab})} \\
        {D(X)} & {D(Y)} & {D(\mathbf{Ab})}
        \arrow["{Ri^!_{Y\!-\!Z,X\!-\!Z}}"', from=1-1, to=1-2]
        \arrow["{R\Gamma^{X-Z}_{Y-Z}}", curve={height=-12pt}, from=1-1, to=1-3]
        \arrow["{R\Gamma_{Y-Z}}"', from=1-2, to=1-3]
        \arrow["{j_{Z,X}^*}"', from=2-1, to=1-1]
        \arrow["{Ri^!_{Y,X}}", from=2-1, to=2-2]
        \arrow["{R\Gamma^X_Z[1]}"', curve={height=12pt}, from=2-1, to=2-3]
        \arrow["{j_{Z,Y}^*}"', from=2-2, to=1-2]
        \arrow["{R\Gamma^Y_Z[1]}", from=2-2, to=2-3]
        \arrow["{\operatorname{id}}"', from=2-3, to=1-3]
    \end{tikzcd}\]
    There are natural transformations corresponding to the boundary maps in the localization sequence:
    \[\partial^Y \colon R\Gamma_{Y-Z} \circ j^*_{Z,Y} \Rightarrow R\Gamma^Y_Z[1] \quad \text{and} \quad \partial^X \colon R\Gamma^{X-Z}_{Y-Z} \circ j^*_{Z,X} \Rightarrow R\Gamma^X_Z[1]\]
    One checks easily that the left square in the diagram is commutative and $Ri^!_{Y,X} \circ \partial^Y = \partial^X$ (details omitted). Hence, there is a morphism of spectral sequences:
    % https://q.uiver.app/#q=WzAsNCxbMCwwLCJcXG9wZXJhdG9ybmFtZXtIfV5wKFlcXCEtWiwgUl5xaV4hXFxtYXRoY2FsIEYpIl0sWzAsMSwiXFxvcGVyYXRvcm5hbWV7SH1ee3ArMX0oWSwgUl5xaV4hXFxtYXRoY2FsIEYpIl0sWzEsMCwiXFxvcGVyYXRvcm5hbWV7SH1ee3ArcX1fe1lcXCEtWn0oWFxcIS1cXCFaLCBcXG1hdGhjYWwgRikiXSxbMSwxLCJcXG9wZXJhdG9ybmFtZXtIfV57cCsxK3F9X1ooWCwgXFxtYXRoY2FsIEYpIl0sWzAsMSwiXFxwYXJ0aWFsXlkiXSxbMCwyLCIiLDIseyJsZXZlbCI6Mn1dLFsxLDMsIiIsMCx7InNob3J0ZW4iOnsic291cmNlIjoyMCwidGFyZ2V0IjoyMH0sImxldmVsIjoyfV0sWzIsMywiXFxwYXJ0aWFsXlgiXV0=
    \[\begin{tikzcd}
        {\operatorname{H}_{\et}^p(Y\!-Z, R^qi^!\mathcal F)} &[-5mm] {\operatorname{H}^{p+q}_{Y\!-Z}(X\!-\!Z, \mathcal F)} \\
        {\operatorname{H}_{Z}^{p+1}(Y, R^qi^!\mathcal F)} & {\operatorname{H}^{p+1+q}_Z(X, \mathcal F)}
        \arrow[Rightarrow, from=1-1, to=1-2]
        \arrow["{\partial^Y}", from=1-1, to=2-1]
        \arrow["{\partial^X}", from=1-2, to=2-2]
        \arrow[shorten <=6pt, shorten >=6pt, Rightarrow, from=2-1, to=2-2]
    \end{tikzcd}\]
    We conclude that the left square in the lemma is commutative. The commutativity of the right square is proven analogously. We leave this verification to the reader.
\end{proof}

Denote by $\operatorname{Cous}(X)$ the complex of Zariski sheaves on $X$ associated with the Cousin complex, and by $\operatorname{\Gamma Cous}(X)$ the corresponding complex of global sections, i.e., the Cousin complex:
\[\operatorname{Cous}(X) = \Big[ 0 \to \bigoplus_{x \in X^{(0)}} i_{x,*}\E_x(X) \to \bigoplus_{x \in X^{(1)}} i_{x,*}\E_x(X) \to \dotsb \Big].\]
We say that the Cousin complex resolves $\E$ locally on $X$ if $\operatorname{Cous}(X)$ is a resolution of the Zariski sheaf $\underline{\E}$ associated with the presheaf $U \mapsto \E(U)$. Theorem~\ref{bo-th} states that the Cousin complex resolves $\E$ locally on $\spec(S)$, where $\spec(S)$ is a geometrically regular over $V$ regular scheme.

Let $U = \spec(S)$ be a local regular $V$-scheme geometrically regular over $V$ of dimension $\geq 2$. In particular, $\pi$ is a local parameter of $S$. Let $f \in S$ be a local parameter such that $(\pi, f)$ forms a regular sequence. Let $U_f = \spec(S_f)$ and let $Z \subset U$ be the vanishing locus of $f$. Then $U_f$ and $Z$ are also regular and geometrically regular over $V$. If $U$ is essentially smooth over $V$, then $U_f$ and $Z$ are also essentially smooth over $V$.

The inclusion $U_f \to U$ induces a surjective morphism $\operatorname{\Gamma Cous}(U) \to \operatorname{\Gamma Cous}(U_f)$. By Lemma~\ref{GysBoundaryCommute} and Gabber's absolute purity theorem, the kernel is naturally isomorphic to $\operatorname{\Gamma Cous}(Z)[-1]$. This yields a short exact sequence:
\[0 \to \operatorname{\Gamma Cous}(Z)[-1] \to \operatorname{\Gamma Cous}(U) \to \operatorname{\Gamma Cous}(U_f) \to 0 \tag{$\star$}\]

Assume that the Cousin complex resolves $\E$ locally on $Z$ and $U_f$. Then $\operatorname{Cous}(U_f)$ is a flasque resolution of the sheaf $\underline{\E}$ on $U_f$, and $\operatorname{\Gamma Cous}(Z)$ is a resolution of $\E(Z)$. The long exact sequence of cohomology associated with $(\star)$ takes the form:
\[0 \to H^0(\operatorname{\Gamma Cous}(U)) \to \underline{\E}(U_f) \xrightarrow{d} \E(Z) \to H^1(\operatorname{\Gamma Cous}(U)) \to \mathrm{H}^1_{\Zar}(U_f, \underline{\E}) \to 0\]
\[0 \to H^i(\operatorname{\Gamma Cous}(U)) \to \mathrm{H}^i_{\Zar}(U_f, \underline{\E}) \to 0 \quad i > 1\]

\begin{lemma}\label{boundaryNat}
    The following diagram is commutative:
    % https://q.uiver.app/#q=WzAsNCxbMCwwLCJcXHVuZGVybGluZVxcRShVX2YpIl0sWzEsMCwiXFxFKFopIl0sWzAsMSwiXFxFKFVfZikiXSxbMSwxLCJcXEVfWihVKSJdLFswLDEsImQiXSxbMiwzLCJcXHBhcnRpYWwiXSxbMiwwLCJcXG9wZXJhdG9ybmFtZXtjYW59Il0sWzMsMSwiXFxvcGVyYXRvcm5hbWV7Z3lzfV57LTF9IiwyXSxbMywxLCJcXHNpbWVxIiwzLHsib2Zmc2V0IjotMywic3R5bGUiOnsiYm9keSI6eyJuYW1lIjoibm9uZSJ9LCJoZWFkIjp7Im5hbWUiOiJub25lIn19fV1d
    \[\begin{tikzcd}
        {\underline\E(U_f)} & {\E(Z)} \\
        {\E(U_f)} & {\E_Z(U)}
        \arrow["d", from=1-1, to=1-2]
        \arrow["{\operatorname{can}}", from=2-1, to=1-1]
        \arrow["\partial", from=2-1, to=2-2]
        \arrow["{\operatorname{gys}^{-1}}"', from=2-2, to=1-2]
        \arrow["\simeq"{marking, allow upside down}, shift left=3, draw=none, from=2-2, to=1-2]
    \end{tikzcd}\]
    Here, $d$ is the boundary map in the long exact sequence associated with $(\star)$, and $\partial$ is the boundary map in the localization sequence.
\end{lemma}
\begin{proof} 
    Omitted, cf. \cite[Lemma~1.7]{PanEqui}. This lemma reflects the fact that differentials in the Cousin complex arise from boundary maps in localization sequences.
\end{proof}

\begin{lemma}\label{cohData}
    Assume that the Cousin complex resolves $\E$ locally on $U_f$ and $Z$. Then the following are equivalent:
    \begin{enumerate}
        \item The Cousin complex resolves $\E$ locally on $U$.
        \item The canonical map $\E(U) \to \E(U_f)$ is injective, the canonical map $\E(U_f) \to \underline{\E}(U_f)$ is an isomorphism, and $\mathrm{H}^i_{\Zar}(U_f, \underline{\E}) = 0$ for $i > 0$.
    \end{enumerate}
\end{lemma}

\begin{proof} 
    Note that if $\E(U) \to H^0(\operatorname{\Gamma Cous}(U))$ is an isomorphism, then $\E(U) \to \E(U_f)$ is injective since $\E(U) \to \E_{\eta}(U) = \E(\eta)$ is injective. Combine the long exact sequence of cohomology associated with $(\star)$, Lemma~\ref{boundaryNat}, and the localization sequence to complete the proof.
\end{proof}

The idea is that we can carry this information across the limit to make the induction step. Recall Popescu's theorem:

\begin{theorem}[D. Popescu, {\cite[Theorem 2.5]{popescu1986general}}]
    If a Noetherian ring $A'$ is geometrically regular over $A$, then it is a filtered inductive limit of finite type smooth $A$-algebras.
\end{theorem}

\begin{proof}[ of Theorem~\ref{bo-th}]
    Assume $U = \spec(S)$ is a local $V$-scheme geometrically regular over $V$. The proof proceeds by induction on $\dim(U)$. The base case is established in \cite[Appendix B]{colliot1997bloch}. Suppose $\dim(U) = n$ and Theorem~\ref{bo-th} holds for dimensions $< n$. Note $\dim(U_f), \dim(Z) < \dim(U)$. As a consequence of Popescu's theorem, we have
    \[S = \varinjlim S^\alpha, \quad S_f = \varinjlim S^\alpha_{f_\alpha}\]
    where $U^\alpha = \spec(S^\alpha)$ is a local regular essentially $V$-smooth scheme, and $f_\alpha$ is a local parameter whose vanishing locus $Z^\alpha$ is also local regular and essentially $V$-smooth (see \cite[Section 3]{PanEqui} for details). Clearly, $U^\alpha_{f_\alpha} = \spec(S^\alpha_{f_\alpha})$ is also regular and essentially $V$-smooth. By Theorem~\ref{geometricCase}, property (1) and the assumptions of Lemma~\ref{cohData} hold for $U^\alpha$, $U^\alpha_{f_\alpha}$, and $Z^\alpha$. Hence, property (2) holds. But property (2) passes through the limit due to the continuity of $\E$ and Grothendieck's limit theorem \cite[Section 6]{PanEqui}. Lemma \ref{cohData} is applied again. The induction step is proved.
\end{proof}

\appendix

\section{Lindel--Ojanguren--Gabber’s lemma over a DVR}

In this section, we provide a proof of the Lindel--Ojanguren--Gabber lemma over DVR. The proof follows the ideas of \cite{panin2024constantcasegrothendieckserreconjecture}. 

We use the following notation throughout this section. Let $R$ be a discrete valuation ring, and let $V = \spec R = \{v, \theta\}$, where $v$ is the closed point and $\theta$ is the generic point. Let $k = k(v) = R/\pi R$ denote the residue field. For a scheme $X$ and $x \in X$, define $x^{(2)} = \spec\left(\mathcal{O}_{X,x}/\mathfrak{m}_{X, x}^2\right)$. If $X$ is a $V$-scheme, let $x_v^{(2)}$ denote the closed fiber of $x^{(2)}$, i.e., $x_v^{(2)} = \spec(\mathcal{O}_{X_v,x}/\mathfrak{m}_{X_v,x}^2)$

\begin{theorem}[cf. {\cite[Theorem 3.5]{panin2024constantcasegrothendieckserreconjecture}}]\label{geomPres} 
    Let $X$ be an irreducible smooth affine $V$-scheme of relative dimension $n \geq 1$, and let $Z \subset X$ be a closed subset of pure codimension 1 containing no irreducible component of $X_v$. Given a finite set of closed points $\underline{x} \subset Z_v \subset X_v$. Suppose the residue field $k(v)$ is infinite or $\underline{x}$ consists of a single point. Then there exists a localization $\mathcal{X}$ of $X$, a localization $\mathcal{A}$ of $\AffSp^n_V$ at a finite set of closed points $\underline{y} \subset \AffSp^n_v$, an element $h \in \mathcal{O}_{\AffSp^n_V, \underline{y}}$, and an elementary distinguished square of the form (i.e. $\mathcal X$ is an \'etale neighbourhood of $\tau(Z)$): 
    % https://q.uiver.app/#q=WzAsNixbMSwwLCJcXG1hdGhjYWwgWCAtIFoiXSxbMSwxLCJcXG1hdGhjYWwgQSAtIFoiXSxbMiwxLCJcXG1hdGhjYWwgQSJdLFsyLDAsIlxcbWF0aGNhbCBYIl0sWzAsMCwiXFxtYXRoY2FsIFhfe1xcdGF1XipofSJdLFswLDEsIlxcbWF0aGNhbCBBX2giXSxbMywyLCJcXHRhdSJdLFswLDEsIlxcdGF1fF97XFxtYXRoY2FsIFggLSBafSJdLFsxLDIsIiIsMSx7InN0eWxlIjp7InRhaWwiOnsibmFtZSI6Imhvb2siLCJzaWRlIjoidG9wIn19fV0sWzAsMywiIiwxLHsic3R5bGUiOnsidGFpbCI6eyJuYW1lIjoiaG9vayIsInNpZGUiOiJ0b3AifX19XSxbNCwwLCI9IiwxLHsibGV2ZWwiOjIsInN0eWxlIjp7ImJvZHkiOnsibmFtZSI6Im5vbmUifSwiaGVhZCI6eyJuYW1lIjoibm9uZSJ9fX1dLFs1LDEsIj0iLDEseyJsZXZlbCI6Miwic3R5bGUiOnsiYm9keSI6eyJuYW1lIjoibm9uZSJ9LCJoZWFkIjp7Im5hbWUiOiJub25lIn19fV0sWzQsNV1d
    \[\begin{tikzcd}
	{\mathcal X_{\tau^*h}} &[-25pt] {\mathcal X - \mathcal Z} & {\mathcal X} \\
	{\mathcal A_h} & {\mathcal A - \tau(\mathcal Z)} & {\mathcal A}
	\arrow["{=}"{description}, draw=none, from=1-1, to=1-2]
	\arrow[from=1-1, to=2-1]
	\arrow[hook, from=1-2, to=1-3]
	\arrow[from=1-2, to=2-2]
	\arrow["\tau", from=1-3, to=2-3]
	\arrow["{=}"{description}, draw=none, from=2-1, to=2-2]
	\arrow[hook, from=2-2, to=2-3]
    \end{tikzcd}\]
    where $\mathcal Z = Z \cap \mathcal X$.
    % "{\tau|_{\mathcal X - \mathcal Z}}", 
\end{theorem}

\begin{lemma}\label{find-point}
    Let $X$ be a smooth variety over a field $k$ of relative dimension $n \geqslant 1$, and let $x \in X$ be a closed point. Then there exists a closed point $y \in \AffSp^n_k$ such that $k[x^{(2)}] \cong k[y^{(2)}]$ as $k$-algebras.
\end{lemma}
\begin{proof}
    There exists a closed point $y \in \AffSp^n_k$ with residue field $k(y) \cong k(x)$. In the case of an infinite field $k$, this is a consequence of Ojanguren’s Lemma \cite[VIII, Corollary 3.2.5]{Knus}. In the case of a finite field $k$, this is a consequence of Lindel’s Lemma \cite[Proposition 2]{Lindel}. Since $X$ is smooth, we may apply the Cohen structure theorem. Hence, the completions of local algebras $\mathcal O_{X, x}$ and $\mathcal O_{\AffSp^n_k, y}$ are isomorphic. In particular, we have the desired isomorphism.
\end{proof}

\begin{proof}[ of Theorem~\ref{geomPres}]
    Note that the question is local around points of $\underline{x}$, so we will sometimes shrink $X$ around $\underline{x}$ (i.e., replacing it with an open neighborhood $X^\circ$). In such cases, we will shrink $Z$ as well (i.e., replacing it with $Z \cap X^\circ$). The proof of the theorem is given in several steps.

    \textit{Step 1}. We start with an appropriate compactification of $X$. Take a morphism $\rho\colon X \to \ProjSp^n_V$ that is quasi-finite at $\underline{x}$. Thus $\rho$ is quasi-finite in some open neighborhood of $\underline{x}$. We shrink $X$ to this neighborhood, so $\rho$ becomes quasi-finite. By Zariski's Main Theorem, $\rho$ factors as an open immersion $X \subset \widebar{X}$ composed with a finite morphism $\widebar{\rho}\colon \widebar{X} \to \ProjSp^n_V$. In particular, $\widebar{X}$ is projective. We can also assume that $\widebar X$ is reduced. Let $\widebar{Z}$ be the closure of $Z$ in $\widebar{X}$. Note that $\widebar{Z}$ contains no irreducible component of $\widebar{X}_v$. 

    Consider the Cohen-Macaulay locus $X_{\operatorname{cm}}$ of $\widebar{X}$. Since $\widebar{X}$ is a scheme of finite type over a DVR, it follows that $X_{\operatorname{cm}}$ is open in $\widebar{X}$ \cite[Ex. 24.2]{Matsumura1989CRT}. Recall that a reduced 1-dimensional Noetherian ring is Cohen-Macaulay. Hence, any codimension $1$ point of $\widebar{X}$ lies in $X_{\operatorname{cm}}$. We conclude that the complement $\widebar{X} - X_{\operatorname{cm}}$ has codimension $\geqslant 2$ in $\widebar{X}$. In particular, $X_{\operatorname{cm}}$ intersects every irreducible component $\widebar{X}_{v,j}$ of $\widebar{X}_v$. Note that $\underline{x} \subset X \subset X_{\operatorname{cm}}$. We replace $X$ with $X_{\operatorname{cm}}$ and $Z$ with $\widebar{Z} \cap X_{\operatorname{cm}}$.
    
    Let $\mathcal{O}_{\widebar{X}}(1)$ be a very ample invertible sheaf on $\widebar{X}/V$ and $s_0 \in \Gamma(\widebar{X}, \mathcal{O}_{\widebar{X}}(1))$ be a section such that $\underline{x} \subset \widebar{X}_{s_0} \subset X$ and $s_0$ does not vanish identically on the irreducible components $\widebar{X}_{v,j}$ of $\widebar{X}_v$ \cite[Lemma 09NV]{stacks-project}. We shrink $X$ again, replacing it with $\widebar{X}_{s_0}$.

    \textit{Step 2}. There exists a finite surjective $V$-morphism $\pi \colon \widebar{X} \to \ProjSp_V(d_0, \dots, d_n)$ to the weighted projective space with $d_0 = 1$ such that:
    \begin{enumerate}
        \item $\pi^{-1}(\pi(\underline{x})) \cap \widebar{Z} = \underline{x}$
        \item $\pi|_{x_v^{(2)}} \colon x_v^{(2)} \to \pi(x)_v^{(2)}$ is a scheme isomorphism for any $x \in \underline{x}$  
    \end{enumerate}
    We construct $\pi$ by choosing sections $s_i$ of sheaves $\mathcal{O}_{\widebar{X}}(d_i)$, $i = 0, \dots, n$:  
    \[\pi = [s_0: \dots : s_n] \colon \widebar{X} \to \ProjSp_V(d_0, \dots, d_n)\]
    Let  
    \[\pi_r = \left(\frac{s_1}{s_0^{d_1}}, \frac{s_2}{s_0^{d_2}}, \dots, \frac{s_r}{s_0^{d_r}}\right) \colon \widebar{X}_{s_0} \to \AffSp^r_V\]
    By Lemma~\ref{find-point}, for each $x \in \underline{x}$, there exists a closed point $y = y(x) \in \AffSp^n_v$ such that $k[x_v^{(2)}] \cong k[y_v^{(2)}]$. Shifting $y$ to a rational vector, we may assume the points $y(x) \in (\mathbb{A}_{v}^1-\{0\})^n \subset \AffSp^n_v$ are pairwise distinct (recall that the residue field is infinite, or $\underline{x}$ consists of a single point). There are generators $f_1^x, \dots, f_n^x$ of the $k(v)$-algebra $k[x_v^{(2)}]$ such that the induced morphism $x_v^{(2)} \to \AffSp^n_v$ is an isomorphism onto $y_v^{(2)}$. We will choose $s_i$ and $d_i$ one by one, monitoring the following assertions:
    \renewcommand{\theenumi}{\roman{enumi}}%
    \begin{enumerate}
        \item $\dim \{s_0 = \dots = s_r = 0\}_v \leqslant n - (r + 1)$
        \item $\dim (\pi_r^{-1}\pi_r(\underline x) \cap \widebar{Z}_v) \leqslant n - (r+1)\;$ %(and $\pi_r^{-1}\pi_r(\underline x) \cap \widebar{Z}_v = \underline x \cap \widebar{Z}_v$ for $r = n$)
        \item $(s_r/s_0^{d_r})|_{x_v^{(2)}} = f_r^x$ for any $x \in \underline x$.
    \end{enumerate}
    For $r=n$ we read assertions $(\mathrm{i})$ and $(\mathrm{ii})$ as \[\{s_0 = \dots = s_n = 0\}_v = \varnothing \quad \text{, and} \quad \pi_n^{-1}\pi_n(\underline x) \cap \widebar{Z}_v = \underline x \cap \widebar{Z}_v\]
    
    Section $s_0 \in \Gamma(\widebar{X}, \mathcal{O}_{\widebar{X}}(1))$ is already constructed in Step 1. We now construct section $s_r$, assuming that sections $s_0, \dots, s_{r-1}$ are already defined. Let $Y_1$ be a finite closed subset of $\{s_0 = \dots = s_{r-1} = 0\}_v$ containing a point in each irreducible component. Let $Y_2$ be a finite closed subset of $\pi_{r-1}^{-1}(\pi_{r-1}(\underline{x})) \cap \widebar{Z}_v$ containing a point in each irreducible component and disjoint from $\underline{x}$ (we ignore irreducible components that consist of a single point $x \in \underline{x}$).
    
    Note that for a closed subscheme $Y \subset \widebar{X}$, the restriction homomorphism $H^0(\widebar{X}, \mathcal{O}_{\widebar{X}}(d)) \to H^0(Y, \mathcal{O}_Y(d))$ is surjective for large $d$. This follows from Serre's vanishing theorem \cite[Theorem III.5.2b]{hartshorne2013algebraic}, which implies $H^1(\widebar{X}, \mathcal{I}_Y(d)) = 0$ for large $d$. Let $Y = Y_1 \sqcup Y_2 \sqcup_{x \in \underline{x}} x_v^{(2)}$, and choose $d_r$ sufficiently large. Thus, there exists a section $s_r \in \Gamma(\widebar{X}, \mathcal{O}_{\widebar{X}}(d_r))$ such that:
    \renewcommand{\theenumi}{\roman{enumi}}%
    \begin{enumerate}
        \item $s_r(y) \neq 0$ for $y \in Y_1$
        \item $s_r(y) = 0$ for $y \in Y_2$
        \item $s_r|_{x_v^{(2)}} = f_r^x s_0^{d_r}|_{x_v^{(2)}}$ for $x \in \underline{x}$
    \end{enumerate}
    Clearly, assertions (i) and (iii) are satisfied. Assertion (ii) holds because $\pi_r(y) \neq \pi_r(x)$ for $x \in \underline{x}$ and $y \in Y_2$ (by construction of $s_r$), so $y \not\in \pi_r^{-1}(\pi_r(\underline{x}))$. Additionally, $\pi_r^{-1}(\pi_r(\underline{x})) \subset \pi_{r-1}^{-1}(\pi_{r-1}(\underline{x}))$.

    By assertion $(\mathrm{i})$, we have a globally defined morphism $\pi$. Since $\{s_0 = \dots = s_n = 0\}$ is a closed subset of the projective $V$-scheme $\widebar{X}$ with empty closed fiber, it is itself empty. Furthermore, $\pi$ is affine and projective, hence finite and surjective by a dimension argument. By assertion $(\mathrm{ii})$, $\pi_n^{-1}(\pi_n(\underline{x})) \cap \widebar{Z}_v = \underline{x}$, so $\pi^{-1}(\pi(\underline{x})) \cap \widebar{Z} = \underline{x}$. Assertion $(\mathrm{iii})$ implies $(2)$.

    \renewcommand{\theenumi}{\arabic{enumi}}%
    \textit{Step 3}. The construction above yields the following corollaries:
    \begin{enumerate}
        \setcounter{enumi}{2}
        \item $X = \pi^{-1}(\AffSp^n_V)$, and $\pi|_X \colon X \to \AffSp^n_V$ is finite and flat.
        \item $\pi$ is \'etale at points of $\underline{x}$.
        \item $\pi$ induces an isomorphism of residue fields $k(\pi(x)) \to k(x)$ for $x \in \underline{x}$.
    \end{enumerate}
    Since $X = \widebar{X}_{s_0}$, we have $X = \pi^{-1}(\AffSp^n_V)$ by construction, and $\pi|_X$ is finite. Flatness of $\pi|_X$ follows from the miracle flatness \cite[Theorem 23.1]{Matsumura1989CRT}. The isomorphism $(2)$ and smoothness of $\widebar{X}_v$ at $\underline{x}$ imply $\pi_v \colon \widebar{X}_v \to \ProjSp_v(d_0, \dots, d_n)$ is unramified at $\underline{x}$. Thus, $\pi$ is unramified (and hence \'etale) at $\underline{x}$ by \cite[Prop. VI.3.6]{AK}. Finally, $(5)$ follows directly from $(2)$.

    \textit{Step 4}. Introduce the following notation:
    \[\underline{y} = \pi(\underline{x}), \quad 
    \mathcal{A} = \spec(\mathcal O_{\AffSp^n_V, \underline{y}}), \quad 
    \mathcal{X} = X \times_{\AffSp^n_V} \mathcal{A}, \quad 
    \mathcal{Z} = Z \times_{\AffSp^n_V} \mathcal{A}.\]
    Consider the induced map $\pi\colon \mathcal{X} \to \mathcal{A}$. Let $\mathcal{D} \subset \mathcal{X}$ denote the branch locus of $\pi$. Define $\mathcal{X}' = \mathcal{X} - \mathcal{D}$, so the restriction $\pi'\colon \mathcal{X}' \to \mathcal{A}$ is étale. Since $\pi$ is étale at $\underline{x}$, we have $\underline{x} \subset \mathcal{X}'$. Also, $\pi$ is open and $\underline{y} = \pi(\underline{x})$ consists of closed points in the semi-local scheme $\mathcal{A}$, hence $\pi$ is surjective. By (1), $\mathcal{Z} \cap \mathcal{D} = \varnothing$, so $\mathcal{Z} \subset \mathcal{X}'$. The finiteness of $\pi$, combined with (1) and (2), implies that $\mathcal O_{\mathcal A, \pi(x)} \to \mathcal{O}_{\mathcal{Z}, x}$ is surjective \cite[Lemma II.7.4]{hartshorne2013algebraic}. Since $\pi(x)$ are distinct for $x \in \underline{x} \subset \mathcal{Z}$, the morphism $\pi|_{\mathcal{Z}} \colon \mathcal{Z} \to \mathcal{A}$ is set-theoretically injective, and the map $\mathcal{O}_{\mathcal{A}} \to \pi_*\mathcal{O}_{\mathcal{Z}}$ is surjective. Thus, $\pi|_{\mathcal{Z}} \colon \mathcal{Z} \to \mathcal{A}$ is a closed embedding. Let $s\colon \pi(\mathcal{Z}) \to \mathcal{Z}$ be the inverse isomorphism, which is a section of the étale morphism $\pi'|_{\pi'^{-1}(\pi(\mathcal{Z}))}$. This gives $\pi'^{-1}(\pi(\mathcal{Z})) = \mathcal{Z} \sqcup \mathcal{Z}'$. Notice $\mathcal Z' \cap \underline x = \varnothing$ since $\underline x \subset \mathcal Z$. Define $\mathcal{X}'' = \mathcal{X}' - \mathcal{Z}'$. We thus obtain an elementary distinguished square:
    % https://q.uiver.app/#q=WzAsNCxbMSwwLCJcXG1hdGhjYWwgWCcnIl0sWzEsMSwiXFxtYXRoY2FsIEEiXSxbMCwxLCJcXG1hdGhjYWwgQSAtIFxcbWF0aGNhbCBcXHBpKFxcbWF0aGNhbCBaKSJdLFswLDAsIlxcbWF0aGNhbCBYJycgLSBcXG1hdGhjYWwgWiJdLFswLDEsIlxccGkiXSxbMiwxLCIiLDAseyJzdHlsZSI6eyJ0YWlsIjp7Im5hbWUiOiJob29rIiwic2lkZSI6InRvcCJ9fX1dLFszLDAsIiIsMix7InN0eWxlIjp7InRhaWwiOnsibmFtZSI6Imhvb2siLCJzaWRlIjoidG9wIn19fV0sWzMsMl1d
    \[\begin{tikzcd}
	{\mathcal X'' - \mathcal Z} & {\mathcal X''} \\
	{\mathcal A - \mathcal \pi(\mathcal Z)} & {\mathcal A}
	\arrow[hook, from=1-1, to=1-2]
	\arrow[from=1-1, to=2-1]
	\arrow["\pi", from=1-2, to=2-2]
	\arrow[hook, from=2-1, to=2-2]
    \end{tikzcd}\]

    \textit{Step 5}. The scheme $\mathcal{X}$ is affine, regular, and semi-local, hence factorial. The branch locus $\mathcal{D}$ is a closed subset of pure codimension 1 \cite[Theorem 6.8]{AK}. Therefore, $\mathcal{X}' = \mathcal{X}_f$ for some $f \in \Gamma(\mathcal{X}, \mathcal{O}_{\mathcal{X}})$. In particular, $\mathcal{X}'$ is also regular and factorial. Similarly, $\mathcal{Z}' \subset \mathcal{X}'$ is a closed subset of pure codimension 1, so $\mathcal{X}'' = \mathcal{X}'_g$. We conclude that $\mathcal{X}''$ is a localization of $X$.

    The scheme $\mathcal{A}$ is affine, regular, and semi-local, hence factorial. The subset $\pi(\mathcal{Z}) \subset \mathcal{A}$ is closed of pure codimension 1, so $\mathcal{A} - \pi(\mathcal{Z}) = \mathcal{A}_h$. Hence, $\mathcal{X}'' - \mathcal{Z} = \mathcal{X}''_{\pi^*h}$.
\end{proof}

\bibliographystyle{acm}
\bibliography{ref}

\end{document}